\documentclass[preprint]{elsarticle}
\journal{a suitable journal}

\usepackage{enumerate}  
\usepackage{amsmath}    
\usepackage{amsthm}     
\usepackage{mathrsfs}   
\usepackage{amsfonts}
\usepackage{amssymb}
\usepackage{mathtools} 
\usepackage{verbatim}	
\usepackage{color}	

\usepackage[T1]{fontenc}

\theoremstyle{plain}                
\newtheorem{thm}{Theorem}[section]  
\newtheorem{prop}[thm]{Proposition}
\newtheorem{cor}[thm]{Corollary}

\theoremstyle{definition}
\newtheorem{defn}[thm]{Definition}

\theoremstyle{remark}
\newtheorem{rem}[thm]{Remark}

\providecommand{\norm}[1]{\lVert#1\rVert}

\def\re{\mathop{\rm Re}\nolimits}
\def\CC{\mathbb C}
\def\RR{\mathbb R}

\def\exp{\mathop{\rm e}\nolimits}
\def\re{\mathop{\rm Re}\nolimits}

\begin{document}

\begin{frontmatter}

\title{On the admissibility of retarded delay systems}

\author[UoL]{Rados{\l}aw Zawiski\corref{correspondingauthor}}
\ead{R.Zawiski@leeds.ac.uk}
\author[UoL]{Jonathan R. Partington}
\cortext[correspondingauthor]{Corresponding author}
\address[UoL]{School of Mathematics, University of Leeds, LS2 9JT Leeds, UK}

\begin{abstract}
We investigate a Hilbert space dynamical system of the form $\dot{z}(t)=Az(t)+A_1z(t-\tau)+Bu(t)$, where $A$ generates a semigroup of contractions and $A_1$ is a bounded operator, in order to determine whether the operator $B$ is admissible. Our approach is based on the Miyadera--Voigt perturbation theorem and the Weiss conjecture on admissibility of control operators for contraction semigroups. We demonstrate that the retarded delay system can be represented as a well-posed abstract Cauchy problem with a solution formed by an initially log-concave bounded semigroup.
\end{abstract}


\end{frontmatter}

\section{Introduction}
In this article we analyse dynamical systems with delay in state variable from the perspective of admissibility of a control operator. The object of our interest is an abstract retarded system
\begin{equation}\label{eqn retarded system introduction}
\left\{\begin{array}{ll}
        \dot{z}(t)=Az(t)+A_{1}z(t-\tau)+Bu(t)\\
        z(0)=z_0,        \\
        \end{array}
\right.
\end{equation}
where, initially, the closed, densely defined operator $A:D(A)\rightarrow X$, $D(A)\subset X$, is a generator of a strongly continuous semigroup $\big(T(t)\big)_{t\geq0}$ such that $T(t)\in\mathcal{L}(X)$ for every $t\geq0$, where $X$ is a Hilbert state space, $B$ is the control operator acting on values of control functions $u\in L^2(J,U)$ with $J$ being a time interval and $U$ a Hilbert space.

A system of the form \eqref{eqn retarded system introduction} without control input $u$ is frequently used as an example of a positive system describing population dynamics and, 
either in an abstract or PDE setting, is well analysed - see \cite[Chapter VI.6]{Engel_Nagel} and references therein.

For a thorough presentation of admissibility results for state-undelayed systems we refer the reader to \cite{Jacob_Partington_2004} and a rich list of references therein. In particular, the results in \cite{Grabowski_Callier_1996} and \cite{Engel_1999} form a basis for considerations in \cite{Batkai_Piazzera} in terms of developing a correct setting, adapted also in this article, in which we conduct the admissibility analysis. The Weiss conjecture, based on which a necessary and sufficient condition for admissibility of hypercontractive semigroups was established in \cite{Jacob_Partington_2001}, was stated in \cite{Weiss_1991}.

In this paper we do not relate \eqref{eqn retarded system introduction} to any particular positive system, but rather treat it as a general starting point for analysis of linear dynamical systems with delay influencing the state vector. We do not put any other assumptions on the undelayed semigroup $\big(T(t)\big)_{t\geq0}$ apart from being a contraction. This allows us to perform analysis in a relatively general case where the necessary and sufficient conditions for admissiblity are known, and yet obtain concrete outcomes. 

The results presented here form a basis for analysis of a specific case of delay system, namely state delay diagonal systems (full details concerning diagonal systems will be presented elswhere \cite{Partington_Zawiski_2018b}).

Section~\ref{sec:2} contains the necessary background results, leading to the main results in Section~\ref{sec:3}.
An example is given in Section~\ref{sec:4}, and some conclusions are given in Section~\ref{sec:5}.

\section{Preliminaries}\label{sec:2}

Apart from definitions introduced in the previous section, throughout this paper the notation $(X,\norm{\cdot}_X)$ and $(U,\norm{\cdot}_U)$ denotes Hilbert spaces with norms coming from appropriate inner products (this is also the case whenever the subscript is omitted). We use the following Sobolev spaces (see \cite{Kreuter_2015} for vector valued functions or \cite[Chapter 5]{Evans} for functionals):
$H^1(J,X)=W^{1,2}(J,X):=\{f\in L^2(J,X):\frac{d}{dt}f(t)\in L^2(J,X)\}$, $H_{c}^1(J,X)=W_{c}^{1,2}(J,X):=\{f\in H^1(J,X):f\vert_{J\setminus S}=0\text{ for every compact }\ S\subset J\}$ and $H_{0}^1(J,X)=W_{0}^{1,2}(J,X):=\{f\in H^1(J,X):\ f(\partial J)=0\}$.
 
\subsection{The state delay equation setting}

To describe a correct setting in which we will consider system \eqref{eqn retarded system introduction}, we follow \cite[Chapter VI.6]{Engel_Nagel} and \cite[Chapter 3.1]{Batkai_Piazzera}. Consider a function  $z:[-\tau,\infty)\rightarrow X$. For each $t\geq0$ we call the function $z_{t}:[-\tau,0]\rightarrow X$, $z_{t}(\sigma):=z(t+\sigma)$, a \textit{history segment} with respect to $t\geq0$. With history segments we consider a function called the \textit{history function} of $z$, that is $h_{z}:[0,\infty)\rightarrow L^2([-\tau,0],X)$, $h_z(t):=z_{t}$. For the whole of the remaining part of this paper we normalize the delay $\tau$ to $\tau=1$. In \cite[Lemma 3.4]{Batkai_Piazzera} we find the following
\begin{prop}\label{prop history function derivative}
Let $z:[-1,\infty)\rightarrow X$ be a function which belongs to \\$H_{loc}^1([-1,\infty],X)$. Then the history function $h_z:t\rightarrow z_t$ of $z$ is continuously differentiable from $\mathbb{R}_+$ into $L^2([-1,0],X)$ with derivative
\begin{equation}\label{eqn history function derivative}
\frac{d}{dt}h_{z}(t)=\frac{d}{d\sigma}z_{t}.
\end{equation}
\end{prop}

Define the Cartesian product $\mathcal{X}:=X\times L^2([-1,0],X)$ with an inner product 
\begin{equation}\label{eqn defn of inner product on Cartesian product}
 \bigg\langle\binom{x}{f},\binom{y}{g}\bigg\rangle_{\mathcal{X}}:=\langle x,y\rangle_{X}+\langle f,g\rangle_{L^2}.
\end{equation}
Then $\mathcal{X}$ becomes a Hilbert space $(\mathcal{X},\|\cdot\|_{\mathcal{X}})$ with the norm $\|\binom{x}{f}\|_{\mathcal{X}}^2=\|x\|_{X}^{2}+\|f\|_{L^2}^{2}$. Consider a linear, autonomous delay differential equation of the form
\begin{equation}\label{eqn delay autonomous diff eq}
\left\{\begin{array}{ll}
        \dot{z}(t)=Az(t)+\Psi z_t\\
        z(0)=x,        \\
        z_0=f,		\\
       \end{array}
\right.
\end{equation}
\color{black}
where $\Psi\in\mathcal{L}(H^1([-1,0],X),X)$ is a \textit{delay operator}, the pair $x\in D(A)$ and $f\in L^2([-1,0],X)$ forms an initial condition. By Proposition~\ref{prop history function derivative} equation \eqref{eqn delay autonomous diff eq} may be written as an abstract Cauchy problem
\begin{equation}\label{eqn defn abstract Cauchy problem}
\left\{\begin{array}{ll}
        \dot{v}(t)=\mathcal{A}v(t)\\
        v(0)=\binom{x}{f},        \\
       \end{array}
\right.
\end{equation} 
where $v:t\rightarrow\binom{z(t)}{z_t}\in\mathcal{X}$ 
and $\mathcal{A}$ is an operator on $\mathcal{X}$ defined as
\begin{equation}\label{eqn defn abstract A}
\mathcal{A}:=\left(\begin{array}{cc} 
								A & \Psi \\ 
								0 & \frac{d}{d\sigma}	
						\end{array}\right),
\end{equation}
with domain
\begin{equation}\label{eqn defn abstract A domain}
D(\mathcal{A}):=\bigg\{\binom{x}{f}\in D(A)\times H^1([-1,0],X):\ f(0)=x\bigg\}.
\end{equation}
The operator $(\mathcal{A},D(\mathcal{A}))$ is closed and densely defined on $\mathcal{X}$ \cite[Lemma 3.6]{Batkai_Piazzera}. Let $\mathcal{A}=\mathcal{A}_0+\mathcal{A}_{\Psi}$, where 
\begin{equation}\label{eqn defn A_0}
\mathcal{A}_0:=\left(\begin{array}{cc} 
								A & 0 \\ 
								0 & \frac{d}{d\sigma}	
						\end{array}\right),\qquad D(\mathcal{A}_0)=D(\mathcal{A}),
\end{equation}
and
\begin{equation}\label{eqn defn  A_psi}
\mathcal{A}_{\Psi}:=\left(\begin{array}{cc} 
								0 & \Psi \\ 
								0 & 0	
						\end{array}\right)\in\mathcal{L}\big(X\times H^1([-1,0],X),\mathcal{X}\big).
\end{equation}
We recall the following Proposition from \cite[Theorem 3.25]{Batkai_Piazzera}, as we will later need the form of a semigroup generated by $(\mathcal{A}_0,D(\mathcal{A}))$.
\begin{prop}\label{prop abstract A_0 as generator}
The following are equivalent:
	\begin{itemize}
		\item[(i)] The operator $(A,D(A))$ generates a strongly continuous semigroup\\ $\big(T(t)\big)_{t\geq0}$ on $X$.
		\item[(ii)] The operator $(\mathcal{A}_0,D(\mathcal{A}_0))$ generates a strongly continuous semigroup\\ $\big(\mathcal{T}_0(t)\big)_{t\geq0}$ on $X\times L^p([-1,0],X)$ for all $1\leq p<\infty$.
	\end{itemize}
	The semigroup $\big(\mathcal{T}_0(t)\big)_{t\geq0}$ is given by 
	\begin{equation}\label{eqn defn T_0 semigroup}
		\mathcal{T}_{0}(t):=\left(\begin{array}{cc} 
									T(t) & 0 \\ 
									S_{t} & S_{0}(t)	
									\end{array}\right)\qquad \forall t\geq0,
	\end{equation}
	where $\big(S_{0}(t)\big)_{t\geq0}$ is the nilpotent left shift semigroup on $L^p([-1,0],X)$, 
	\begin{equation}
	S_{0}(t)f(\tau):=\left\{\begin{array}{ll}
        								f(\tau+t)	&	\hbox{if} \ \tau+t\in[-1,0],\\
        								0		&	\hbox{otherwise}\\
       									\end{array}\right.
	\end{equation}
	and $S_{t}:X\rightarrow L^p([-1,0],X)$, 
	\begin{equation}
	(S_{t}x)(\tau):=\left\{\begin{array}{ll}
        								T(\tau+t)x	&	\hbox{if}\ -t<\tau\leq0,\\
        								0					&	\hbox{if} -1\leq\tau\leq-t.\\
       									\end{array}\right.
	\end{equation}
\end{prop}

In order to make use of the Miyadera--Voigt Perturbation Theorem, we need the following

\begin{defn}\label{defn Sobolev tower}
Let $\beta\in\rho(A)$ and denote $(X_1,\norm{\cdot}_1):=(D(A),\norm{\cdot}_1)$ with $\norm{\cdot}_1:=\norm{(\beta I-A)x}\ (x\in D(A))$ .
 
Similarly, we set $\norm{x}_{-1}:=\norm{(\beta I-A)^{-1}x}\ (x\in X)$. Then the space $(X_{-1},\norm{\cdot}_{-1})$ denotes the completion of $X$ under the norm $\norm{\cdot}_{-1}$. For $t\geq0$ we define $T_{-1}(t)$ as the continuous extension of $T(t)$ to the space  $(X_{-1},\norm{\cdot}_{-1})$. 
\end{defn}
In the sequel, much of our reasoning is justified by the following proposition, to which we do not refer directly but 
include here for the reader's convenience. 
\begin{prop}\label{prop Hilbert rigged space}
With notation of Definition \ref{defn Sobolev tower} we have the following
\begin{itemize}
 \item[(i)] The spaces $(X_1,\norm{\cdot}_1)$ and $(X_{-1},\norm{\cdot}_{-1})$ are independent of the choice of $\beta\in\rho(A)$.
 \item[(ii)] $(T_1(t))_{t\geq0}$ is a strongly continuous semigroup on the Banach space \\$(X_1,\norm{\cdot}_1)$ and we have $\norm{T_1(t)}_1=\norm{T(t)}$ for all $t\geq0$.
 \item[(iii)] $(T_{-1}(t))_{t\geq0}$ is a strongly continuous semigroup on the Banach space \\$(X_{-1},\norm{\cdot}_{-1})$ and we have $\norm{T_{-1}(t)}_{-1}=\norm{T(t)}$ for all $t\geq0$.
\end{itemize}
\end{prop}
See \cite[Chapter II.5]{Engel_Nagel} or \cite[Chapter 2.10]{Tucsnak_Weiss} for more details on these elements. A sufficient condition for $P\in\mathcal{L}(X_1,X)$ to be a perturbation of Miyadera-Voigt class, 
hence implying that $A+P$ is a generator on $X$, takes the form of  \cite[Corollaries III.3.15 and 3.16]{Engel_Nagel} 

\begin{prop}\label{prop Miyadera-Voigt sufficient condition of well posedness}
Let $(A,D(A))$ be the generator of a strongly continuous semigroup $\big(T(t)\big)_{t\geq0}$ on a Banach space $X$ and let ${P\in\mathcal{L}(X_1,X)}$ be a perturbation which satisfies
\begin{equation}\label{eqn condition on perturbation for well-posedness}
\int_{0}^{t_0}\norm{PT(r)x}dr\leq q\Vert x\Vert\qquad\forall x\in D(A)
\end{equation}
for some $0\leq q<1$. Then
 the sum $A+P$ with domain $D(A+P):=D(A)$ generates a strongly continuous semigroup $(S(t))_{t\geq0}$ on $X$. Moreover, for all $t\geq0$ the semigroup $(S(t))_{t\geq0}$ satisfies
\begin{equation}\label{eqn perturbed semigroup in integral form}
 S(t)x=T(t)x+\int_{0}^{t}S(s)PT(t-s)xds\qquad\forall x\in D(A).
\end{equation}
\end{prop}

\subsection{The admissibility problem}

The basic object in the formulation of the admissibility problem is a linear system and its mild solution
\begin{equation}\label{eqn basic object}
 \frac{d}{dt}x(t)=Ax(t)+Bu(t);\quad x(t)=T(t)x_0+\int_{0}^{t}T(t-s)Bu(s)ds,
\end{equation}
where $x:[0,\infty)\rightarrow X$, $u\in V$ where $V$ is a space of measurable functions from $[0,\infty)$ to $U$ and $B$ is a \textit{control operator}; $x_0\in X$ is an initial state.

In many practical examples the control operator $B$ is unbounded. In such cases \eqref{eqn basic object} is viewed on an extrapolation space $X_{-1}\supset X$, where $B\in\mathcal{L}(U,X_{-1})$. To ensure that the state $x(t)$ lies in $X$ it is sufficient that $\int_{0}^{t}T_{-1}(t-s)Bu(s)ds\in X$ for all inputs $u\in V$. Put differently, we have the following

\begin{defn}
 The control operator $B\in\mathcal{L}(U,X_{-1})$ is said to be \textit{finite-time admissible} for a semigroup $\big(T(t)\big)_{t\geq0}$ on a Hilbert space $X$ if for each $\tau>0$ there is a constant $c(\tau)$ such that the condition 
\begin{equation}\label{eqn defn finite-time admissibility by norm inequality}
 \Big\|\int_{0}^{\tau}T(\tau-s)Bu(s)ds\Big\|_X\leq c(\tau)\norm{u}_V
\end{equation}
holds for all inputs $u$, and an \textit{infinite-time admissible} if the condition \eqref{eqn defn finite-time admissibility by norm inequality} holds for all $\tau>0$ with $c(\tau)$ uniformly bounded.
\end{defn}

For contraction semigroups the following proposition was shown in \cite{Jacob_Partington_2001}:

\begin{prop}\label{prop admissibility condition}
Let $\big(T(t)\big)_{t\geq0}$ be a $C_0$-semigroup of contractions on a separable Hilbert space $X$ with infinitesimal generator $A$ and let $B\in\mathcal{L}(U,X_{-1})$, where $\dim U < \infty$. Then $B$ is infinite-time admissible if and only if there exists a constant $C>0$ such that the following resolvent condition holds
\begin{equation}\label{eqn admissibility condition for contraction semigroups}
\|(\lambda I-A)^{-1}B\| \le \frac{C}{\sqrt{\re \lambda}} \qquad \forall\lambda \in \CC_+.
\end{equation}
\end{prop}

\begin{rem}\label{rem infinite_finite time admissibility}
Condition \eqref{eqn admissibility condition for contraction semigroups}, which is usually easier to check than admissibility itself has as a consequence the following observation that if the semigroup satisfies $\|T(t)\| \le e^{\omega t}$, so that $A-\omega I$ generates a contraction semigroup, then finite-time admissibility for the pair $(A,B)$ follows from the resolvent condition
\begin{equation}\label{eqn admissibility condition for exponentially bounded semigroups}
\|( \lambda I-A)^{-1}B\| \le \frac{C}{\sqrt{\re \lambda-\omega}} \qquad \forall \re \lambda > \omega.
\end{equation}
\end{rem}

The next result is a useful tool \cite{Kreuter_2015} in many norm estimations:

\begin{thm}[Sobolev Embedding Theorem]\label{thm Sobolev Embedding Theorem}
Let $X$ be a Banach space and $1\leq p\leq\infty$, then there exists a constant $C$ such that
\[
 \norm{f}_{L^{\infty}(J,X)}\leq C\norm{f}_{W^{1,p}(J,X)}
\]
for all $f\in W^{1,p}(J,X)$, i.e. the embedding $W^{1,p}(J,X)\hookrightarrow L^\infty(J,X)$ is continuous. Further, the inclusion $W^{1,p}(J,X)\subset C_b(J,X)$ holds, where $C_b(J,X)$ is the space of all continuous and bounded functions from $J$ to $X$ with the supremum norm.
 
\end{thm}

\section{Retarded non-autonomous dynamical systems}\label{sec:3}
We begin with an analysis of   retarded non-autonomous dynamical systems of the form 
\begin{equation}\label{eqn retarded non-autonomous system}
\left\{\begin{array}{ll}
        \dot{z}(t)=Az(t)+\Psi z_t+Bu(t)\\
        z(0)=x,        \\
        z_0=f,		\\
       \end{array}
\right.
\end{equation}
where all the elements are as in \eqref{eqn delay autonomous diff eq}, $u\in L^2(0,\infty;U)$, $B$ is a control operator and the delay operator $\Psi\in\mathcal{L}(H^1([-1,0],X),X)$,

\begin{equation}\label{eqn delay operator example}
\Psi(f):=A_{1}f(-1),
\end{equation}
with $A_{1}\in\mathcal{L}(X)$. Note that a generalization to the case $\Psi(f):=\Sigma_{k=1}^{n}A_kf(-h_k)$ with $f\in H^1([-1,0],X)$, $A_k\in\mathcal{L}(X)$ and $h_k\in[0,1]$ for each $k=1,\dots,n$ is straightforward and will be omitted.

Following the procedure described in the Preliminaries section, for the system \eqref{eqn retarded non-autonomous system} we define a non-autonomous abstract Cauchy problem 

\begin{equation}\label{eqn defn non-autonomous abstract Cauchy problem}
\left\{\begin{array}{ll}
        \dot{v}(t)=\mathcal{A}v(t)+\mathcal{B}u(t)\\
        v(0)=\binom{x}{f},        \\
       \end{array}
\right.
\end{equation} 
which we consider firstly on the space $\mathcal{X}$ with $\mathcal{B}=\binom{B}{0}$, and then on its completion $\mathcal{X}_{-1}$ where the control operator $\mathcal{B}\in\mathcal{L}(U,\mathcal{X}_{-1})$. 

The delay operator $\Psi$ defined in \eqref{eqn delay operator example} is an example of a much wider class of delay operators, with which condition \eqref{eqn condition on perturbation for well-posedness} is satisfied and $(\mathcal{A},D(\mathcal{A}))$ remains a generator of a strongly continuous semigroup (see \cite[Chapter 3.3.3]{Batkai_Piazzera}). Hence \eqref{eqn defn non-autonomous abstract Cauchy problem} is well-posed and we can formally write its $\mathcal{X}_{-1}$-valued mild solution as

\begin{equation}\label{eqn k-th component abstract Cauchy problem mild solution}
 v(t)=\mathcal{T}(t)v(0)+\int_{0}^{t}\mathcal{T}(t-s)\mathcal{B}u(s)ds,
\end{equation}
where $\mathcal{T}(t)\in\mathcal{L}(\mathcal{X}_{-1})$ is the extension of the semigroup generated by $(\mathcal{A},D(\mathcal{A}))$, where the latter semigroup is given by the implicit formula \eqref{eqn perturbed semigroup in integral form}. The remining part is to find the space $\mathcal{X}_{-1}$. We begin with determination of the adjoint $\mathcal{A}^*$, with a reasoning similar to \cite[Chapter A.3.64]{Curtain_Zwart}.
\begin{prop}\label{prop adjoint of abstract A}
Let $(\mathcal{A},D(\mathcal{A}))$ be as defined by \eqref{eqn defn abstract A} and \eqref{eqn defn abstract A domain}. Then its adjoint operator $\mathcal{A}^*$ is given by

\begin{equation}\label{eqn defn adjoint of abstract A}
\mathcal{A}^*:=\left(\begin{array}{cc} 
								A^* & 0 \\ 
								\Psi^* & -\frac{d}{d\sigma}	
						\end{array}\right),
\end{equation}
with domain
\begin{equation}\label{eqn defn adjoint of abstract A domain}
D(\mathcal{A}^*):={D(A^*)\times H_{0}^1([-1,0],X)}.
\end{equation}
\end{prop}
\begin{proof}
From $(A,D(A))$ being closed we have $D(\mathcal{A}^*)\neq\emptyset$. Let now $v=\binom{x}{f}\in D(\mathcal{A})$ and $w=\binom{y}{g}\in D(\mathcal{A}^*)$ and denote $A_{0}:=\frac{d}{d\sigma}$. Then, by the definition of adjoint operator 
	\begin{equation}\label{eqn calculation of adjoint of abstract A domain}
	\begin{split}
	\langle\mathcal{A}v,w\rangle_{\mathcal{X}}&=\langle Ax+\Psi f,y\rangle_{X}+\langle A_{0}f,g\rangle_{L^2}\\
	&=\langle x,A^*y\rangle_{X}+\langle f,\Psi^*y+A_{0}^*g\rangle_{L^2}=\langle v,\mathcal{A}^*w\rangle_{\mathcal{X}}.
	\end{split}
\end{equation}
The calculation in \eqref{eqn calculation of adjoint of abstract A domain} is correct provided that $D(\mathcal{A}^*)$ is defined in an appropriate way. Namely, assuming that $D(A^*)$ is properly defined, we have to examine only the term
\begin{equation}\label{eqn prop adjoint of abstract A calc 1}
	\begin{split}
	&\langle f,\Psi^*y+A_{0}^*g\rangle_{L^2}=\langle f,\Psi^*y-\frac{d}{dt}g\rangle_{L^2}=\int_{-1}^{0}\Big\langle f(t),\Psi^*y(t)-\frac{d}{dt}g(t)\Big\rangle_{X}dt\\
	&=\int_{-1}^{0}\Big\langle f(t),\Psi^*y(t)\Big\rangle_{X}dt-\Big\langle f(t),g(t)\Big\rangle_{X}\Big\vert_{-1}^0+\int_{-1}^{0}\Big\langle \frac{d}{dt}f(t),g(t)\Big\rangle_{X}dt.
	\end{split}
\end{equation}
Since $\Psi\in\mathcal{L}\big(H^1([-1,0],X),X\big)$ and $X$ is a Hilbert space, the domain of $\Psi^*$ is $D(\Psi^*)=X$, where we identify $X$ with its dual $X'$. This results in
\[
\int_{-1}^{0}\Big\langle f(t),\Psi^*y(t)\Big\rangle_{X}dt=\langle \Psi f,y\rangle_{X}\int_{-1}^{0}dt=\langle \Psi f,y\rangle_{X}.
\]
The remaining term of \eqref{eqn prop adjoint of abstract A calc 1} is
\begin{equation}\label{eqn prop adjoint of abstract A calc 2}
\begin{split}
&-\Big\langle f(t),g(t)\Big\rangle_{X}\Big\vert_{-1}^0+\int_{-1}^{0}\Big\langle \frac{d}{dt}f(t),g(t)\Big\rangle_{X}dt\\
&=\langle f(-1),g(-1)\rangle_{X}-\langle f(0),g(0)\rangle_{X}+\int_{-1}^{0}\Big\langle \frac{d}{dt}f(t),g(t)\Big\rangle_{X}dt\\&=\langle A_{0}f,g\rangle_{L^2}
\end{split}
\end{equation}
if and only if 
\begin{equation}\label{eqn prop adjoint of abstract A calc 3}
\langle f(-1),g(-1)\rangle_{X}-\langle x,g(0)\rangle_{X}=0\qquad\forall v\in D(\mathcal{A}).
\end{equation}
As $x$ and $f$ need only to be in $D(\mathcal{A})$, for $g$  to satisfy \eqref{eqn prop adjoint of abstract A calc 3} for every $\binom{x}{f}\in D(\mathcal{A)}$ it has to be $g\in H_{0}^1([-1,0],X)$. As $D(A^*)\subset X$ densely and $H_{0}^1([-1,0],X)\subset L^2([-1,0],X)$ densely \cite{Evans} we obtain $D(\mathcal{A}^*):=D(A^*)\times H_{0}^1([-1,0],X).$
\end{proof}
Due to the fact that $\mathcal{X}_{-1}$ is the dual to $D(\mathcal{A}^*)$ with respect to the pivot space $\mathcal{X}$, we may explicitly write 
\begin{equation}\label{eqn the extended space X-1}
 \mathcal{X}_{-1}=\big(X_{-1}\times H^{-1}([-1,0],X)\big)
\end{equation}
where $H^{-1}([-1,0],X)$ is the dual to $H_{0}^{1}([-1,0],X)$ with respect to the pivot space $L^2([-1,0],X)$ - see \cite[Chapter 2.10 and Definition 13.4.7]{Tucsnak_Weiss}.

\subsection{Contraction semigroups}

As our main tool for admissibility analysis is expressed in Proposition \ref{prop admissibility condition}, it is important to see if the delay semigroup $(\mathcal{T}(t))_{t\geq0}$ is hypercontractive, i.e., $\norm{\mathcal{T}(t)}\leq e^{\omega t}$ for every $t\geq0$ and some $\omega\in\mathbb{R}$. In the case when the operator $(A,D(A))$ in the retarded system \eqref{eqn retarded non-autonomous system} generates a contraction semigroup we start the analysis with the following 
\begin{prop}\label{prop T_0 semigroup uniform bound}
Let $\big(T(t)\big)_{t\geq0}$ be a semigroup of contractions generated by $(A,D(A))$. Then the semigroup $\big(\mathcal{T}_0(t)\big)_{t\geq0}$ generated by $(\mathcal{A}_0,D(\mathcal{A}_0))$ is hypercontractive and
\[
\norm{\mathcal{T}_{0}(t)}\leq e^{\frac{1}{2}t}\qquad\forall t\geq0. 
\] 
\end{prop}
\begin{proof}
 Fix $t>0$ and $v=\binom{x}{f}\in\mathcal{X}$. We can calculate
\begin{equation}\label{eq:22}
 \|\mathcal{T}_{0}(t)v\|_{\mathcal{X}}^2=\norm{T(t)x}_{X}^{2}+\norm{S_{t}x}_{L^2}^{2}+2\re\langle S_{t}x,S_{0}(t)f\rangle_{L^2}+\norm{S_{0}(t)f}_{L^2}^{2}.
\end{equation}
The second term of \eqref{eq:22} expands to
\begin{equation*}
\norm{S_{t}x}_{L^2}^2=\int_{-t}^{0}\langle T(\tau+t)x,T(\tau+t)x\rangle_{X}d\tau=\int_{0}^{t}\norm{T(\tau)x}_{X}^{2}d\tau,
\end{equation*}
while the fourth one expands to
\begin{align*}
\norm{S_{0}(t)f}_{L^2}^2&=\int_{-1}^{0}\langle (S_{0}(t)f)(\tau),(S_{0}(t)f)(\tau)\rangle_{X}d\tau=
\int_{-1+t}^{0}\langle f(\tau),f(\tau)\rangle_{X}d\tau
\end{align*}
if $t\in[0,1]$ and $\norm{S_{0}(t)f}_{L^2}=0$ if $t>1$. As for the third term note that according to the definition $(S_{t}x)(\tau)=0$ for $\tau\in[-1,-t]$ and $(S_{0}(t)f)(\tau)=0$ for $\tau\in(-t,\infty)$. Hence, 
\begin{align*}
2\re\langle S_{t}x,S_{0}(t)f\rangle_{L^2}=2\re\int_{-1}^{0}\langle (S_{t}x)(\tau),(S_{0}(t)f)(\tau)\rangle_{X}d\tau=0
\end{align*}
for all $t\geq0$. The contraction assumption now gives the following estimation
\begin{align*}
&\|\mathcal{T}_{0}(t)v\|_{\mathcal{X}}^2\leq\norm{T(t)x}_{X}^{2}+\int_{0}^{t}\norm{T(\tau)x}_{X}^{2}d\tau+\int_{-1+t}^{0}\langle f(\tau),f(\tau)\rangle_{X}d\tau\\
&\leq\norm{x}_{X}^{2}+t\norm{x}_{X}^{2}+\norm{f}_{L^2}^{2}\leq (1+t)(\norm{x}_{X}^{2}+\norm{f}_{L^2}^{2})<e^{t}\norm{v}_{\mathcal{X}}^2.
\end{align*}
\end{proof}

Proposition \ref{prop T_0 semigroup uniform bound} opens up a wide field of applications of perturbation and approximation of semigroups results. 
We will continue to follow the Miyadera--Voigt approach given in Proposition \ref{prop Miyadera-Voigt sufficient condition of well posedness}.

\begin{prop}\label{prop T semigroup uniform bound}
Let $\big(T(t)\big)_{t\geq0}$ be the semigroup of contractions generated by $(A,D(A))$, $\big(\mathcal{T}_0(t)\big)_{t\geq0}$ be the semigroup generated by $(\mathcal{A}_0,D(\mathcal{A}_0))$ and suppose that $(\mathcal{A}_{\Psi},D(\mathcal{A}_{\Psi}))$ is the perturbing operator. Then for the semigroup $\big(\mathcal{T}(t)\big)_{t\geq0}$ generated by $\big(\mathcal{A}_0+\mathcal{A}_\Psi,D(\mathcal{A}_0)\big)$ the inequality
\begin{equation}\label{eqn T semigroup norm bound}
 \norm{\mathcal{T}(t)}\leq e^{\frac{1}{2}t}(1+\norm{A_1}Mt^{\frac{1}{2}})\quad\forall t\in[0,1]
\end{equation}
holds, where $A_1$ comes from~\eqref{eqn delay operator example} and $M\leq\sqrt{2}\exp^{2\norm{A_1}^2}$.

\end{prop}

\begin{proof}
\begin{itemize}
\item[1.] From Proposition \ref{prop Miyadera-Voigt sufficient condition of well posedness} the semigroup $\big(\mathcal{T}(t)\big)_{t\geq0}$ is given by
\begin{equation}\label{eqn T semigroup in integral form}
 \mathcal{T}(t)v=\mathcal{T}_{0}(t)v+\int_{0}^{t}\mathcal{T}(s)\mathcal{A}_\Psi\mathcal{T}_{0}(t-s)vds\qquad\forall t\geq0,\ \forall v\in D(\mathcal{A}_0).
\end{equation}

Due to Proposition~\ref{prop T_0 semigroup uniform bound} the operator $(\mathcal{A}_0-\frac{1}{2}\mathcal{I},D(\mathcal{A}_0))$ generates a contraction semigroup $\big(\mathcal{T}_1(t)\big)_{t\geq0}$ on $\mathcal{X}$, where $\mathcal{T}_1(t)=e^{-\frac{1}{2}t}\mathcal{T}_0(t)$ forall $t\geq0$. In consequence, the operator $(\mathcal{A}_0+\mathcal{A}_\Psi-\frac{1}{2}\mathcal{I},D(\mathcal{A}_0))$ generates a semigroup $\big(\mathcal{T}_r(t)\big)_{t\geq0}$ on $\mathcal{X}$ where $\mathcal{T}_r(t)=e^{-\frac{1}{2}t}\mathcal{T}(t)$ for all $t\geq0$. Equation \eqref{eqn T semigroup in integral form} for the rescaled semigroup $\mathcal{T}_r(t)$ and $v=\binom{x}{f}\in D(\mathcal{A}_0)$ reads
\begin{equation}\label{eqn semigroup Tr explicit equation}
 \begin{split}
 \mathcal{T}_r(t)v=&\mathcal{T}_1(t)v+\int_{0}^{t}\mathcal{T}_r(s)\mathcal{A}_\Psi e^{-\frac{1}{2}(t-s)}\mathcal{T}_{0}(t-s)vds\\
		  =&\mathcal{T}_1(t)v+\int_{0}^{t}e^{-\frac{1}{2}(t-s)}\mathcal{T}_r(s)
\left(\begin{array}{c} 
									\Psi(S_{t-s}x)+\Psi(S_{0}(t-s)f)\\ 
									0	
									\end{array}\right)ds\\
\end{split}
\end{equation}

\item[2.] Before estimating the norm of $\mathcal{T}_r(t)$ consider the following 
\begin{equation}\label{eqn estimation of Psi T_0 on v}
 \begin{split}
  &\norm{\Psi(S_{t-s}x)+\Psi(S_{0}(t-s)f}_X\\
  &=\norm{A_1(S_{t-s}x)(-1)+A_1\big(S_{0}(t-s)f\big)(-1)}_X\\
  &\leq\norm{A_1T(-1+t-s)x)}_X+\norm{A_1f(-1+t-s)}_X\\
  &\leq\norm{A_1}\norm{x}_X+\norm{A_1}\norm{f(-1+t-s)}_X,
  \end{split}
\end{equation}
where we denote by the same symbol a continuous bounded representative of $f\in H^1([-1,0],X)$. Because of the Sobolev Embedding Theorem \ref{thm Sobolev Embedding Theorem} we know that such representative exists. Due to the H{\"o}lder inequality and again the Sobolev Embedding Theorem we have also
\begin{equation}\label{eqn estimation by Holder of int f}
 \begin{split}
  &\int_{0}^{t}\norm{f(-1+t-s)}_Xds=\int_{-1}^{-1+t}\norm{f(s)}_Xds\\
  &\leq t^{\frac{1}{2}}\Big(\int_{-1}^{-1+t}\norm{f(s)}^2_Xds\Big)^{\frac{1}{2}}\leq t^{\frac{1}{2}}\norm{f}_{L^2}.
 \end{split}
\end{equation}

\item[3.] Fix $v=\binom{x}{f}\in D(\mathcal{A}_0)$ and let $t\in[0,1]$. Using above results we have 
\begin{align*}
 &\norm{\mathcal{T}_r(t)v}\leq\norm{\mathcal{T}_1(t)v}+\int_{0}^{t}\norm{\mathcal{T}_r(s)\mathcal{A}_\Psi e^{-\frac{1}{2}(t-s)}\mathcal{T}_{0}(t-s)v}ds\\
  &\leq\norm{\mathcal{T}_1(t)v}+\int_{0}^{t}\norm{\mathcal{T}_r(s)}\Big(\norm{A_1}\norm{x}_X+\norm{A_1}\norm{f(-1+t-s)}_X\Big)ds\\
  &\leq\norm{v}+\norm{A_1}\norm{x}_X\int_{0}^{t}\norm{\mathcal{T}_r(s)}ds+\norm{A_1}\int_{0}^{t}\norm{\mathcal{T}_r(s)}\norm{f(-1+t-s)}_Xds\\
  &\leq\norm{v}+\norm{A_1}\norm{x}_{X}Mt+\norm{A_1}M\int_{0}^{t}\norm{f(-1+t-s)}_Xds\\
  &\leq\norm{v}+\norm{A_1}M\big(t\norm{x}_{X}+t^\frac{1}{2}\norm{f}_{L^2}\big)\leq\big(1+\norm{A_1}Mt^{\frac{1}{2}}\big)\norm{v}
\end{align*}
where $M:=\max\big\{\norm{\mathcal{T}_r(s)}:s\in[0,1]\big\}$.

\item[4.] Consider a square of norm estimation resulting from \eqref{eqn semigroup Tr explicit equation}, namely
\begin{align*}
 &\norm{\mathcal{T}_r(t)v}^2=\Big(\norm{\mathcal{T}_1(t)v}+\int_{0}^{t}\norm{\mathcal{T}_r(s)\mathcal{A}_\Psi e^{-\frac{1}{2}(t-s)}\mathcal{T}_{0}(t-s)v}ds\Big)^2\\
 &\leq2\norm{\mathcal{T}_1(t)v}^2+2\bigg(\int_{0}^{t}\norm{\mathcal{T}_r(s)}\Big(\norm{A_1}\norm{x}_X+\norm{A_1}\norm{f(-1+t-s)}_X\Big)ds\bigg)^2\\
 &\leq2\norm{v}_{\mathcal{X}}^2+4\norm{A_1}^2\norm{x}_X^2\Big(\int_{0}^{t}\norm{\mathcal{T}_r(s)}ds\Big)^2+\\
 &\quad+4\norm{A_1}^2\Big(\int_{0}^{t}\norm{\mathcal{T}_r(s)}\norm{f(-1+t-s)}_Xds\Big)^2\\
 &\leq2\norm{v}_{\mathcal{X}}^2+4\norm{A_1}^2\Big(t\norm{x}_X^2+\norm{f}_{L^2}^2\Big)\int_{0}^{t}\norm{\mathcal{T}_r(s)}^2ds,
\end{align*}
where we used the H{\"o}lder inequality twice. As $t\in[0,1]$ we have $t\norm{x}_X^2+\norm{f}_{L^2}^2\leq\norm{x}_X^2+\norm{f}_{L^2}^2=\norm{v}_{\mathcal{X}}^2$
 and the above estimation gives
\begin{equation}\label{eqn rescaled semigroup Tr for Gronwall lemmat}
 \norm{\mathcal{T}_r(t)}^2\leq2+4\norm{A_1}^2\int_{0}^{t}\norm{\mathcal{T}_r(s)}^2ds.
\end{equation}
The Gr{\"o}nwall--Bellman lemma (see, for example, \cite{Pachpatte} or \cite{Dragomir} for an exposition of such  inequalities) now gives
\[
 \norm{\mathcal{T}_r(t)}^2\leq2\exp^{4\norm{A_1}^2t}\qquad\forall t\in[0,1].
\]
Hence,
\begin{equation}\label{eqn rescaled semigroup Tr second norm bound}
 \norm{\mathcal{T}_r(t)}\leq\sqrt{2}\exp^{2\norm{A_1}^2t}\qquad\forall t\in[0,1],
\end{equation}
and we obtain that $M\leq\sqrt{2}\exp^{2\norm{A_1}^2}$.

\item[5.] Getting back to the original delay semigroup $\mathcal{T}(t)$ we finish the proof.
\end{itemize}
\end{proof}

\begin{cor}\label{cor log-concave bound}
Under assumptions of Proposition~\ref{prop T semigroup uniform bound} the rescaled semigroup $\mathcal{T}_r(t)$ is initially log-concave bounded, that is there exists $v:[0,1]\rightarrow[0,\infty)$ such that $v(t):=\log(N(t))\geq\log(\norm{\mathcal{T}_r(t)})$ for some function $N:[0,1]\rightarrow\RR_+$. \qed
\end{cor}

With Proposition \ref{prop admissibility condition} we may state a necessary and sufficient condition for finite time admissibility of the retarded system given by \eqref{eqn retarded non-autonomous system}, namely

\begin{thm} \label{thm main thm}
Using the previously defined notation for the retarded non-autonomous dynamical system \eqref{eqn retarded non-autonomous system} let the control operator $\mathcal{B}:=\binom{B}{0}\in\mathcal{L}(U,\mathcal{X}_{-1})$, where $\dim U < \infty$, and there exist $\eta>0$ and  $\omega<\infty$ such that the inequality
\begin{equation}\label{eqn T hypercontractive semigroup norm bound}
\norm{\mathcal{T}(t)}\leq \exp^{\omega t}\quad\forall t\in(0,\eta)
\end{equation}
 holds. Then the control operator $\mathcal{B}$ is finite-time admissible if and only if there exists a constant $C>0$ such that the following resolvent condition holds
\[
\|( \lambda \mathcal{I-A})^{-1}\mathcal{B}\| \le \frac{C}{\sqrt{\re \lambda-\omega}} \qquad \forall \re \lambda > \omega.
\]

\end{thm}
\begin{proof}
The proof of this theorem follows from Proposition~\ref{prop admissibility condition}, Remark~\ref{rem infinite_finite time admissibility}, Proposition~\ref{prop T semigroup uniform bound} and semigroup property.
\end{proof}

Note that \eqref{eqn T hypercontractive semigroup norm bound} in Theorem~\ref{thm main thm} does not follow from Proposition~\ref{prop T semigroup uniform bound}. The necessary and sufficient condition for \eqref{eqn T hypercontractive semigroup norm bound} to hold is
\begin{equation}\label{eqn iff condition of hypercontractivity by inner product}
\re\langle\mathcal{A}v,v\rangle_{\mathcal{X}}\leq\omega \quad\forall v\in D(\mathcal{A}).
\end{equation}
Under the relatively weak assumptions made by us (in fact in Theorem~\ref{thm main thm}, as in this whole subsection, we assume only the contraction property of the undelayed semigroup $T(t)$ and a simple form of the delay operator $\Psi$) condition \eqref{eqn iff condition of hypercontractivity by inner product} takes the form
\begin{equation}
\re\langle A_1f(-1),f(0)\rangle_{X}\leq\omega \quad\forall \binom{x}{f}\in D(\mathcal{A}),
\end{equation}
and whether one can draw conclusions on hypercontractivity under such weak assumptions remains an open problem.

A natural way of strengthening the result of Proposition~\ref{prop T semigroup uniform bound} and thus Theorem~\ref{thm main thm} would be to add a condition on the differentiability of $\mathcal{T}:[0,\eta)\rightarrow(\mathcal{L}(\mathcal{X}),\norm{\cdot}_{\mathcal{L}(\mathcal{X})})$ in the form
\begin{equation}\label{eqn bound on T derivative}
\limsup_{t\rightarrow0^+}\frac{d}{dt}\norm{\mathcal{T}(t)}_{\mathcal{L}(\mathcal{X})}<\infty.
\end{equation}
However, the question of what properties the undelayed semigroup $T(t)$ must have so that the conclusion \eqref{eqn bound on T derivative} can be drawn, remains open in the setting of this paper.

A noticeable fact is that Corollary~\ref{cor log-concave bound} says that the set of log-concave bounds for the delayed semigroup $\mathcal{T}(t)$ is not empty. Hence, another way one may look at the hypercontractivity problem is given in \cite{Davies_2005} by means of the upper log-concave envelope of $\norm{\mathcal{T}(t)}$.

\color{black}
\section{Example}\label{sec:4}
As an example of a retarded dynamical system consider a Lotka--Scharpe or the McKendrick--von Foerster equation
as in \cite[Example 3.16]{Batkai_Piazzera}. In general, it may be seen as describing a population aging with delay, where the delay can be a result of measuring time or cell development. 
\begin{equation}\label{eqn example pde with delay}
\left\{\begin{array}{ll}
        \partial_tz(t,s)+\partial_sz(t,s)=-\mu(s)z(t,s)+
\nu(s)z(t-1,s),& t\geq0,\ s\in\mathbb{R}_+\\
        z(t,0)=\int_{0}^{\infty}\beta(r)z(t,r)dr, &t\geq0,        \\
        z(t,s)=f(t,s),& (t,s)\in[-1,0]\times\mathbb{R}_+,\\
       \end{array}
\right.
\end{equation}
where $\mu,\nu,\beta\in L^\infty(\mathbb{R}_+)$; $\mu,\beta$ are positive
and $f$ is in $H^1([-1,0]\times\mathbb{R}_+)$. In the abstract setting we may specify:
\begin{itemize}
 \item the Banach space $X:=L^2(\mathbb{R}_+),$
 \item the operator $(Ag)(s):=-g'(s)-\mu(s)g(s),\ s\in\mathbb{R}_+$ with the domain \\
       $D(A):=\big\{g\in H^1(\mathbb{R}_+):g(0)=\int_{0}^{\infty}\beta(r)g(r)dr\big\};$\\
we see that $A$ generates a contraction semigroup by applying the perturbation result in \cite[Theorem III.2.7]{Engel_Nagel}.
\item the delay operator $\Psi:H^1([-1,0],X)\rightarrow X$ defined as $\Psi(f):=\nu f(-1).$
\end{itemize}
 With the above definitions we obtain an autonomous abstract system, to which we can apply a suitable control signal and obtain a well-posed abstract Cauchy problem \eqref{eqn defn abstract Cauchy problem} representing a system of the form \eqref{eqn retarded non-autonomous system}.

\section{Conclusions}\label{sec:5}
The admissibility analysis of the retarded delay system with bounded $A_1$ operator presented in this paper is a good starting point in the admissibility analysis of other state-delayed systems. In our future work particular attention among such systems will be paid to systems which have a well-known structure giving additional insight, such as diagonal systems.

The  admissibility results obtained here are also a natural starting point for the analysis of controllability or observability of state-delayed systems.

\paragraph{Acknowledgements}
This project has received funding from the European Union's Horizon 2020 research and innovation programme under the Marie Sk{\l}odowska-Curie grant agreement No 700833.\\

\centerline{\bf References}

\bibliographystyle{amsplain}

\bibliography{delay_systems_library}

\end{document}